\documentclass[12pt]{amsart}
\usepackage{amsmath,amsfonts,amsthm,amscd,amssymb,mathrsfs,amssymb}
\usepackage{stmaryrd}
\usepackage{graphicx}
\newtheorem{theorem}{Theorem}

\newtheorem{proposition}[theorem]{Proposition}
\newtheorem{lemma}[theorem]{Lemma}
\newtheorem{definition}[theorem]{Definition}

\newtheorem{remark}[theorem]{Remark}

\newcommand{\CP}{\mathbb{CP}}
\newcommand{\CC}{\mathbb{C}}

\newcommand{\RR}{\mathbb{R}}

\numberwithin{equation}{section}
\numberwithin{theorem}{section}
\numberwithin{table}{section}
\numberwithin{table}{section}
\begin{document}
\bibliographystyle{amsalpha} 
\title[Einstein metrics]{Einstein metrics and Yamabe invariants of weighted projective spaces}
\author{Jeff A. Viaclovsky}
\address{Department of Mathematics, University of Wisconsin, Madison, 
WI, 53706}
\email{jeffv@math.wisc.edu}
\thanks{Research partially supported by NSF Grants DMS-0804042 and DMS-1105187}
\dedicatory{This article is dedicated to the memory of Friedrich Hirzebruch}
\begin{abstract}An orbifold version of the Hitchin-Thorpe inequality is used to prove
that certain weighted projective spaces do not admit orbifold Einstein metrics.
Also, several
estimates for the orbifold Yamabe invariants of 
weighted projective spaces are proved.
\end{abstract}
\date{June 6, 2012.}
\maketitle
\vspace{-5mm}
\section{Introduction}
\label{intro}

 This article is concerned with certain orbifolds
in dimension four with isolated singularities modeled on $\RR^4 / \Gamma$,
where $\Gamma$ is a finite subgroup of ${\rm{SO}}(4)$ 
acting freely on $\RR^4 \setminus \{0 \}$. 
The examples considered are
 weighted projective spaces:
\begin{definition} {\em{ For relatively prime integers $1 \leq r \leq q \leq p$, 
the {\em{weighted projective space}} $\mathbb{CP}^2_{(r,q,p)}$ is $S^{5}/S^1$, 
where $S^1$ acts by
\begin{align}
(z_0,z_1,z_2)\mapsto (e^{ir\theta}z_0,e^{iq\theta}z_1 ,e^{ip\theta}z_2),
\end{align}
for $0\leq \theta <2\pi$. 
}}
\end{definition}
The weighted projective space $\mathbb{CP}^2_{(r,q,p)}$ has no singular 
points if and only if $(r,q,p) = (1,1,1)$. In general, the orbifold group at
each singular point is a cyclic group, with action described below in Section \ref{osg}. 

A {\em{Riemannian metric}} on an orbifold is a smooth 
Riemannian metric away from the singular set, such that near any singular point
the metric locally lifts to a smooth $\Gamma$-invariant metric on $B^4$. 
\subsection{Einstein metrics}
The first result is the following non-existence theorem.
 \begin{theorem} 
\label{wpsthm}
If $p > 1$, then the weighted projective space
$\mathbb{CP}^2_{(r,q,p)}$ does not admit any
K\"ahler-Einstein metric with respect to any complex structure. 
Furthermore, if 
\begin{align}
\label{mainineq}
p \geq ( \sqrt{q} + \sqrt{r})^2,
\end{align}
then the weighted projective space 
$\mathbb{CP}^2_{(r,q,p)}$ does not admit {\em{any}}  Einstein metric.
\end{theorem}

\begin{remark}
{\em Assuming the complex structure is standard, 
non-existence of a K\"ahler-Einstein metric on weighted projective 
spaces for $p > 1$ was shown in previous works \cite{Mabuchi, GMSY, RossThomas}. 
It is emphasized that the non-existence proof given in this paper does not make any assumptions about the complex structure. 
}
\end{remark}
Robert Bryant proved that every weighted projective space admits a Bochner-K\"ahler metric
\cite{Bryant}, and subsequently, David and Gauduchon gave an alternative construction  
and showed that this metric is the unique Bochner-K\"ahler metric on a given weighted 
projective space \cite[Appendix D]{DavidGauduchon}.
Consequently, this metric will be called the {\em{canonical}} Bochner-K\"ahler metric.  
It is noted that this metric is the quotient of a Sasakian structure on 
$S^5$ under the $S^1$-action, which implies that it is an orbifold Riemannian 
metric in the above sense.

Note that in real dimension four,  
Bochner-K\"ahler metrics are the same as self-dual K\"ahler metrics.
Derdzinski \cite{Derdzinski} proved that for 
self-dual K\"ahler metric $g$, the conformal metric 
$\tilde{g}=R_g^{-2}g$ is a self-dual Hermitian {\em{Einstein}} metric,
away from the zero set of the scalar curvature $R_g$. 
This conformal metric is not K\"ahler unless 
$R_g$ is a constant. 

For a weighted projective space $\mathbb{CP}^2_{(r,q,p)}$ with Bochner-K\"ahler metric
$g$, the zero set of the scalar curvature is easily identified 
using \cite[(2.32)]{DavidGauduchon}, which yields the following 3 cases: 
\begin{itemize}
\item
If $p < r+q$, then $R_{g}>0$ everywhere, and
$\tilde{g}$ is a positive Einstein metric.  
\item
If $p = r+q$, then $R_{g}> 0$ except at one point, and
$\tilde{g}$ is Ricci-flat away from this point.  
\item If $p > r+q$, $R_{g}$ vanishes 
along a hypersurface and the complement consists of two open sets 
on which $\tilde{g}$ has negative Einstein constant.  
\end{itemize}
\begin{remark}{\em
In relation to Theorem \ref{wpsthm}, 
$\tilde{g}$ is a global Einstein metric in the case $p < r + q$, but
the author does not know if there exists an 
Einstein metric on $\mathbb{CP}^2_{(r,q,p)}$ in the range 
$ r + q  \leq p < ( \sqrt{q} + \sqrt{r})^2$; this is a
very interesting problem. 
}
\end{remark}

 The main tools used in proving Theorem \ref{wpsthm} 
are an orbifold version of the Hitchin-Thorpe 
inequality \cite{Hitchin1, Thorpe} and the 
triple reciprocity law for Dedekind sums 
of Rademacher \cite{Rademacher1}.
Similar computations for the signature were previously done 
by Hirzebruch and Zagier \cite{HirzebruchZagier, Zagier}.
For another recent application 
of this reciprocity law, see \cite{LockViaclovsky}.

The weighted projective space $\mathbb{CP}^2_{(1,1,p)}$ is 
the one-point compactification of $\mathcal{O}(-p)$, 
the complex line bundle over $\CP^1$, which will be denoted by 
$\widehat{\mathcal{O}(-p)}$ (noting that 
$\mathcal{O}(-p)$ is diffeomorphic to $\mathcal{O}(p)$).
The above theorem in this special case is then simply as follows.
\begin{theorem}
If $p \geq 4$ then 
$\widehat{\mathcal{O}(-p)}$ does not admit 
any Einstein metric. 
\end{theorem}
The case $p = 1$ is just $\CP^2$ which of course admits
an Einstein metric, the Fubini-Study metric. 
The author does not know if either $\widehat{\mathcal{O}(-2)}$ 
or $\widehat{\mathcal{O}(-3)}$ admits an Einstein 
metric. Exactly as above, $\mathcal{O}(-2)$ does admit 
a complete Ricci-flat Einstein metric, the well-known 
Eguchi-Hanson metric \cite{EguchiHanson}, but this does 
not yield an Einstein metric on the compactification
$\widehat{\mathcal{O}(-2)}$. 

\subsection{Orbifold Yamabe invariants}
The next results deal with orbifold Yamabe invariants
(see \cite{AB2} for background on the orbifold Yamabe problem).
The conformal orbifold Yamabe invariant is defined by
\begin{align}
Y_{\mathrm{orb}}(M,[g]) = \inf_{\tilde{g} \in [g]} \mathrm{Vol}( \tilde{g})^{-1/2}
\int_M R_{\tilde{g}} dV_{\tilde{g}}, 
\end{align}
where $[g]$ denotes the conformal class of $g$. 
The orbifold Yamabe invariant is then defined as 
\begin{align}
Y_{\mathrm{orb}} (M) = \sup_{[g]} Y_{\mathrm{orb}}(M,[g]),
\end{align}
where the supremum is taken over all conformal classes. 

If $M$ is a weighted projective space satisfying $1 \leq r \leq q \leq p$, 
then since $p$ is the size of the largest orbifold group, 
any conformal class satisfies the estimate 
\begin{align}
\label{elemest}
Y_{\mathrm{orb}}(M,[g]) \leq \frac{8 \pi \sqrt{6}}{\sqrt{p}}.
\end{align}
This follows from \cite[Corollary 2.10]{AB2}, and will be called
the {\em{elementary estimate}} of Akutagawa-Botvinnik. 

The main estimate for the orbifold Yamabe invariants of 
weighted projective spaces is the following:
\begin{theorem}
\label{t4}
If $M = \mathbb{CP}^2_{(r,q,p)}$, then 
\begin{align}
\label{uest}
Y_{\mathrm{orb}}(M) \leq 4 \pi \sqrt{2} \frac{(r + q + p)}{\sqrt{rqp}},
\end{align}
and if 
\begin{align}
\label{in3}
p < ( \sqrt{r} + \sqrt{q})^2, 
\end{align}
then the lower estimate
\begin{align}
\label{oyf1}
 Y_{\mathrm{orb}}(M) \geq 4 \pi \sqrt{6} 
\sqrt{ \frac{2}{r} + \frac{2}{q} + \frac{2}{p}
- \frac{r}{pq} -\frac{q}{pr}- \frac{p}{qr}}  
\end{align}
is satisfied. Furthermore, 
if $r + q \leq p < ( \sqrt{r} + \sqrt{q})^2$
then strict inequality holds in~\eqref{oyf1}.
\end{theorem}
The upper and lower estimates on 
the Yamabe invariant in Theorem \ref{t4} 
coincide only for  $(p,q,r) = (1,1,1)$. 
In this case, the Fubini-Study metric is a {\em{supreme}} 
Einstein metric, using terminology of LeBrun \cite{LeBrunEinstein}.
In the case $p < q + r$,  the lower bound in \eqref{oyf1} is 
in fact the Yamabe energy of 
the Einstein metric $\tilde{g}$. 
Interestingly, the upper bound in \eqref{uest} turns out to
be the Yamabe energy of the canonical Bochner-K\"ahler metric.
However, for $p > 1$, this is {\em{not}} a Yamabe minimizer in its conformal 
class; it does not even have constant scalar curvature.
The upper estimate in  \eqref{uest} is likely not sharp;
except for the Fubini-Study metric, 
the upper bound in \eqref{uest} is not attained by any conformal class:
\begin{theorem}
\label{t5}
If $M = \mathbb{CP}^2_{(r,q,p)}$ and $p > 1$, then any conformal class $[g]$ satisfies
\begin{align}
\label{uestconf}
Y_{\mathrm{orb}}(M, [g]) < 4 \pi \sqrt{2} \frac{(r + q + p)}{\sqrt{rqp}}.
\end{align}
\end{theorem}
Note that in case 
\begin{align}
4 \pi \sqrt{2} \frac{(r + q + p)}{\sqrt{rqp}} > \frac{8 \pi \sqrt{6}}{\sqrt{p}},
\end{align}
Theorem \ref{t5} is trivial and follows 
from the elementary estimate \eqref{elemest}. 
However, there are many cases when the upper bound in \eqref{uestconf} 
is strictly smaller than the elementary estimate (see 
Theorem \ref{ulthm} below). 

The proof of \eqref{oyf1} follows more or less immediately from the 
Hitchin-Thorpe inequality on orbifolds used to prove Theorem \ref{wpsthm}.
However, the proof of \eqref{uest} is more subtle, and 
follows the idea of Gursky-LeBrun \cite{GurskyLeBrun} adapted to 
orbifolds by Akutagawa-Botvinnik \cite{AB2}.  For convenience, 
a slightly simplified proof of this result is given in Section~\ref{YI},
which is also used to prove Theorem \ref{t5}.
In \cite{AB2}, the estimate \eqref{uest}
was applied to the example of $\mathcal{O}(-p)$ (the case
of $\CP^2_{(1,1,p)}$), but the upper estimate \eqref{uest} 
is not ``effective'' for $p > 1$ since 
\eqref{uest} is larger than the elementary estimate \eqref{elemest}
in that case. 
So it is only interesting when the upper estimate given 
in \eqref{uest} is strictly smaller than the elementary estimate \eqref{elemest}. 
This turns out to hold for a large class of weighted projective spaces:
\begin{theorem}Let $M = \mathbb{CP}^2_{(r,q,p)}$, with $1 \leq r \leq q \leq p$. 
\label{ulthm}
If 
\begin{align}
\label{pin}
p < (2 \sqrt{3} - 3) q + r \sim 0.464 q + r, 
\end{align}
then 
\begin{align}
\label{oyf2}
0 <  4 \pi \sqrt{6} 
\sqrt{ \frac{2}{r} + \frac{2}{q} + \frac{2}{p}
- \frac{r}{pq} -\frac{q}{pr}- \frac{p}{qr}} 
\leq
Y_{\mathrm{orb}}(M)
 \leq  4 \pi \sqrt{2} \frac{(r + q + p)}{\sqrt{rqp}} < \frac{8 \pi \sqrt{6}}{\sqrt{p}}. 
\end{align}
\end{theorem}
To conclude, it is remarked that 
only a few orbifold Yamabe invariants are known exactly.
For example, in \cite{ViaclovskyFourier} it was shown that  
the orbifold conformal compactification of a hyperk\"ahler ALE metric
in dimension four has maximal orbifold Yamabe invariant. 
That argument also gives an exact determination of the
orbifold Yamabe invariant in the ``critical'' case $p = q + r$:
\begin{theorem}
\label{crithm}
Let $M = \mathbb{CP}^2_{(r,q,p)}$, and let $g$ be the canonical 
Bochner-K\"ahler metric. If $p = q + r$,  then 
there is no constant scalar curvature metric in the 
conformal class of $g$, and 
\begin{align}
\label{maxest}
Y_{\mathrm{orb}}(M,[g]) =  \frac{8 \pi \sqrt{6}}{\sqrt{p}}.
\end{align}
Consequently, 
\begin{align}
\label{ycrit}
Y_{\mathrm{orb}}(M) = \frac{8 \pi \sqrt{6}}{\sqrt{p}}.
\end{align}
\end{theorem}
The proof of this result is  based on the Obata argument \cite{Obata},
and is more or less is the same as \cite[Theorem 1.3]{ViaclovskyFourier},
with a few minor modifications. 
\begin{remark}{\em
The author does not know if the orbifold Yamabe problem has a solution 
if $p > r + q$ on $\mathbb{CP}^2_{(r,q,p)}$ in the conformal 
class of the canonical Bochner-K\"ahler metric $g$. However, {\em{symmetric}} solutions 
were ruled out in the case $(1,1,p)$ in \cite[Theorem 1.4]{ViaclovskyFourier}.  
}
\end{remark}


\subsection{Acknowledgments}
The author would like to thank Michael T. Lock for 
very useful discussions, and 
Xiaodong Wang for assistance with the argument in Theorem~\ref{gthm}.
The author would also like to thank the anonymous referee who made numerous helpful suggestions to improve the exposition of the paper. 
\section{Einstein metrics}
\label{ogi}

 Let $(M,g)$ be a Riemannian orbifold with singular points $x_i$, 
$i = 1 \dots N$. 
The Euler characteristic is given by 
\begin{align}
\label{euler}
\chi(M)=\frac{1}{8\pi^2}\int_{M}\Big(|W|^2-\frac{1}{2}|E|^2+\frac{1}{24}R^2\Big)dV_{g}
+\sum_{i =1}^N\frac{|\Gamma_i|-1}{|\Gamma_i|},
\end{align}
where $E$ denotes the traceless Ricci tensor $E = Ric - (R/4) g$, 
and the signature is given by 
\begin{align}
\label{tau}
\tau(M)=\frac{1}{12\pi^2}\int_{M}(|W^+|^2-|W^-|^2)dV_{g}- \sum_{i=1}^N \eta(S^3/\Gamma_i),
\end{align}
where $\Gamma_i \subset {\rm{SO}}(4)$ 
is the orbifold group around the point $p_i$,
and $\eta(S^3/\Gamma_i)$ is the eta-invariant.
See \cite{Hitchin, Nakajima} for a discussion of the 
formulas $\eqref{euler}$ and $\eqref{tau}$.

\subsection{Cyclic group actions}
For $1 \leq q < p$ relatively prime integers, denote by $\Gamma_{(q,p)}$ the 
cyclic action 
\begin{align}
\label{qpaction}
\left(
\begin{matrix}
\exp^{2 \pi i k / p}   &  0 \\
0    & \exp^{2 \pi i k q / p } \\
\end{matrix}
\right),  \ \ 0 \leq k < p,
\end{align}
acting on $\RR^4 \simeq \CC^2$. The action  
$\Gamma_{(q,p)}$ will be referred to as a type $(q,p)$-action.
If $\Gamma_i$ is conjugate to a $\Gamma_{(q,p)}$ action in ${\rm{SO}}(4)$,  
then
\begin{align}
\eta(S^3/\Gamma_i) = 4 s(q,p),
\end{align}
where
\begin{align}
\label{Dedekind}
s(q,p) = \frac{1}{4p}\sum_{j=1}^{p-1}\Big[\cot(\frac{\pi}{p}j)\cot(\frac{\pi}{p}qj)\Big]
\end{align}
is the well-known Dedekind sum \cite{APS}.

\subsection{Weighted projective spaces}
\label{osg}
For relatively prime integers $a < b$, let 
$a^{-1;b}$  denote the inverse of $a$ modulo $b$.
On $\mathbb{CP}^2_{(r,q,p)}$ there are three possible orbifold 
points:
\begin{enumerate}
\item $[1,0,0]$ with a type $(q^{-1;r}p,r)$-action.
\item $[0,1,0]$ with a type $(p^{-1;q}r,q)$-action.
\item $[0,0,1]$ with a type $(r^{-1;p}q,p)$-action.
\end{enumerate}
Consequently, on a weighted projective space, the Chern-Gauss-Bonnet 
formula is 
\begin{align}
\label{eulerwps}
\chi(M)=\frac{1}{8\pi^2}\int_{M}\Big( |W|^2 -\frac{1}{2}|E|^2 + \frac{1}{24}R^2 \Big)dV_{g}
+ \bigg[  \frac{|r|-1}{|r|} +  \frac{|q|-1}{|q|} +
\frac{|p|-1}{|p|} \bigg].
\end{align}
Since $\chi(M) = 3$ (see \cite{KawasakiWPS}), this may be rewritten as 
\begin{align}
\label{eulerwps2}
\frac{1}{8\pi^2}\int_{M}\Big( |W|^2 -\frac{1}{2}|E|^2 + \frac{1}{24}R^2 \Big)dV_{g}
= \frac{1}{r} + \frac{1}{q} + \frac{1}{p}.
\end{align}

Next, on a weighted projective space, the Hirzebruch signature formula is
\begin{align}
\label{tauwps}
\tau(M)=\frac{1}{12\pi^2}\int_{M}(|W^+|^2-|W^-|^2)dV_{g}- 
4 \bigg[ s(q^{-1;r}p,r) + s(p^{-1;q}r,q) + s(r^{-1;p}q,p) \bigg].
\end{align}
Rademacher's triple reciprocity for Dedekind sums \cite{Rademacher1}
\begin{align}
s(q^{-1;r}p,r)+s(p^{-1;q}r,q)+s(r^{-1;p}q,p) = -\frac{1}{4}+\frac{1}{12}\bigg(\frac{r}{pq}+\frac{q}{pr}+\frac{p}{qr}\bigg),
\end{align}
implies that
\begin{align}
\label{tauwps2}
\tau(M)
& = \frac{1}{12\pi^2}\int_{M}(|W^+|^2-|W^-|^2)dV_{g} + 1
- \frac{1}{3}\bigg(\frac{r}{pq}+\frac{q}{pr}+\frac{p}{qr}\bigg).
\end{align}
Since $\tau(M) = 1$ (see \cite{KawasakiWPS}), this may be rewritten as 
\begin{align}
\label{tauwps3}
\frac{1}{12\pi^2}\int_{M}(|W^+|^2-|W^-|^2)dV_{g}
= \frac{1}{3}\bigg(\frac{r}{pq}+\frac{q}{pr}+\frac{p}{qr}\bigg).
\end{align}

The following argument to rule out K\"ahler-Einstein metrics for $p > 1$  
is an adaptation of the argument of 
\cite[Lemma 3]{Derdzinski} to weighted projective spaces:
\begin{theorem}
\label{noKE}
Let $M = \mathbb{CP}^2_{(r,q,p)}$. Then $M$ admits a 
K\"ahler-Einstein metric if and only if $(r,q,p) = (1,1,1)$. 
\end{theorem}
\begin{proof}
Any K\"ahler metric satisfies
\begin{align}
\label{wplus}
|W^+|^2 = \frac{R^2}{24}.
\end{align}
Consequently, the Gauss-Bonnet formula for any K\"ahler metric on $M$ is
\begin{align}
\label{eulerkahler}
\frac{1}{8\pi^2}\int_{M}\Big( 2|W^+|^2 + |W^-|^2 -\frac{1}{2}|E|^2 \Big)dV_{g}
= \frac{1}{r} + \frac{1}{q} + \frac{1}{p}.
\end{align}
Subtracting \eqref{eulerkahler} from $3$ times \eqref{tauwps3} yields
\begin{align}
\begin{split}
-\frac{3}{8 \pi^2} \int_{M} |W^-|^2dV_{g} + \frac{1}{16 \pi^2} \int_M |E|^2 dV_g
&= \frac{r}{pq}+\frac{q}{pr}+\frac{p}{qr} - \frac{1}{r} - \frac{1}{q} - \frac{1}{p}\\
& = \frac{1}{r q p} ( r^2 + q^2 + p^2 - pq - pr - qr)\\
& = \frac{1}{2 r q p} \big(  (p-r)^2 + (p-q)^2 + (q - r)^2 \big). 
\end{split}
\end{align}
Consequently, if $g$ is K\"ahler-Einstein, this gives a contradiction 
since the left-hand side is nonpositive and the right-hand side 
is strictly positive unless 
$(p,q,r) = (1,1,1)$ in which case the Fubini-Study metric is a 
K\"ahler-Einstein metric. 
\end{proof}
The following theorem is a generalization of the Hitchin-Thorpe 
inequality \cite{Hitchin1, Thorpe} to weighted projective spaces: 
\begin{theorem} 
\label{wpsthmbody}
If 
\begin{align}
\label{mainineqbody}
p \geq ( \sqrt{q} + \sqrt{r})^2,
\end{align}
then the weighted projective space 
$\mathbb{CP}^2_{(r,q,p)}$ does not admit {\em{any}}  Einstein metric.
\end{theorem}
\begin{proof}
Subtracting $3$ times \eqref{tauwps3} from $2$ times \eqref{eulerwps2} yields
\begin{align}
\label{htwps}
\frac{1}{4\pi^2} \int_M \Big(  2 |W^-|^2 - \frac{1}{2}|E|^2 + \frac{1}{24}R^2 \Big)
dV_g 
= \frac{2}{r} + \frac{2}{q} + \frac{2}{p} - \frac{r}{pq} -\frac{q}{pr} - \frac{p}{qr},
\end{align}
for any metric $g$.  Next, assume that $g$ is an Einstein 
metric on $M = \mathbb{CP}^2_{(r,q,p)}$.
Then \eqref{htwps} yields the inequality 
\begin{align}
\frac{r}{pq}+\frac{q}{pr}+\frac{p}{qr} \leq  \frac{2}{r} + \frac{2}{q} + \frac{2}{p},
\end{align}
whereupon multiplication by $pqr$ results in the inequality
\begin{align}
r^2 + q^2 + p^2 \leq 2 ( pq + pr + qr),
\end{align}
which is rewritten as 
\begin{align}
\label{qp}
p^2 - 2 (q + r) p + (q-r)^2 \leq 0. 
\end{align}
For fixed $q$ and $r$, consider the left-hand side of the 
above equation as a quadratic polynomial in $p$. 
By the quadratic formula, the roots are
\begin{align}
p_{\pm} = q + r \pm 2 \sqrt{qr}, 
\end{align}
Clearly then, the inequality in \eqref{qp} is satisfied if
\begin{align}
p_- = ( \sqrt{q} - \sqrt{r})^2 \leq p \leq ( \sqrt{q} + \sqrt{r})^2 = p_+.
\end{align}
Since $1 \leq r \leq q \leq p$, it follows that
\begin{align}
p_- = q + r - 2 \sqrt{qr} \leq q + r - 2 r = q - r < q,
\end{align}
so the lower inequality is already satisfied. Consequently,  the only requirement
is that
\begin{align}
  p \leq ( \sqrt{q} + \sqrt{r})^2 = p_+.
\end{align}
In the case of equality $p = p_+$, from \eqref{htwps}, 
the metric must be Ricci-flat and self-dual, so 
the bundle $\Lambda^2_-$ is flat. 
Since $\mathbb{CP}^2_{(r,q,p)}$  is simply connected,
$\Lambda^2_-$ must be trivial, and 
the holonomy reduces to ${\rm{SU}}(2)$. The metric 
is therefore K\"ahler with zero Ricci tensor, which 
contradicts Theorem~\ref{noKE}. 
\end{proof}
Theorem \ref{wpsthm} immediately follows from Theorems \ref{noKE} and \ref{wpsthmbody}.

\section{Orbifold Yamabe invariants}
\label{YI}
The following Proposition is a restatement of Theorem~\ref{t5}, and 
immediately implies 
the upper estimate on the orbifold Yamabe invariant in Theorem~\ref{t4}.
The proof is based on the idea of Gursky-LeBrun \cite{GurskyLeBrun} adapted to 
orbifolds by Akutagawa-Botvinnik \cite{AB2}.
\begin{proposition}
\label{gthm}
 If $g$ is any Riemannian metric on $M = \mathbb{CP}^2_{(r,q,p)}$,
then 
\begin{align}
\label{c1}
Y_{\mathrm{orb}}(M, [g])\leq  4 \pi \sqrt{2} \frac{(r + q + p)}{\sqrt{rqp}}.
\end{align}
Furthermore, if $p > 1$, then strict inequality holds in \eqref{c1}. 
\end{proposition}
\begin{proof}
First, one may assume that $g$ has positive scalar 
curvature. Let $L$ be the $\mathrm{Spin}^c$ structure associated to the almost 
complex structure $J$ on $M$, and let $D$ denote the Dirac operator:
\begin{align}
D:\Gamma(S^+)\rightarrow \Gamma(S^+).
\end{align}
From \cite[Theorem 2]{Fukumoto}, it follows that $\mathrm{Ind}(D) = 1$. 
Therefore, there exists a positive harmonic spinor $\psi\ne 0$.
By the Lichnerowicz-Bochner formula,
\begin{align}\label{boc}
\nabla^*\nabla\psi+\frac{R}{4}\psi+\frac{1}{2}F^+\cdot \psi = 0,
\end{align}
where $F$ is the curvature form of the line bundle and one chooses the
connection such that $F$ is a harmonic $2$-form. Pairing
this with $\psi$ and using the Kato inequality
\begin{align}
|\nabla\psi|^2\geq \frac{4}{3}|\nabla|\psi||^2,
\end{align}
yield
\begin{align}
\frac{1}{2}\Delta |\psi|^2\geq \frac{4}{3}|\nabla |\psi||^2
           +\frac{R}{4}|\psi|^2+\frac{1}{2}\langle F^+\cdot \psi,\psi\rangle.
\end{align}
It follows from the Cauchy-Schwarz inequality $|\langle F^+\cdot \psi,\psi\rangle|\leq \sqrt{2}|F^+||\psi|^2$ that
\begin{align}
|\psi|\Delta|\psi|\geq \frac{1}{3}|\nabla|\psi||^2+\frac{R}{4}|\psi|^2
  -\frac{\sqrt{2}}{2}|F^+||\psi|^2.
\end{align}
Letting $u=|\psi|^{2/3}$, it follows that
\begin{align}
-\Delta u+\frac{R}{6}u\leq \frac{\sqrt{2}}{3}|F^+|u.
\end{align}
Multiplying by $6 \cdot u$ and integrating by parts, 
\begin{align}
\frac{\int (6 |\nabla u|^2+ R u^2) dV}{(\int u^4 dV)^{1/2}}
\leq 2\sqrt{2}\big(\int |F^+|^2 dV\big)^{1/2}.
\end{align}
Since $b^2_-= 0$, $c_1(L)=\frac{\sqrt{-1}}{2\pi}F$,  which yields
\begin{align}
Y_{\mathrm{orb}}(M, [g])\leq 4 \pi \sqrt{2} \big( \int_M c_1(L)^2\big)^{1/2}.
\end{align}
Since $L$ is the anti-canonical bundle, the first Chern class 
satisfies $c_1(L) = c_1 (M)$, and
from elementary complex geometry, $p_1(M) = c_1(M)^2 - 2 c_2(M)$. 
By Chern-Weil theory and \eqref{eulerwps2} and \eqref{tauwps3} above, 
\begin{align*}
\begin{split}
\int_M c_1(L)^2 &= \frac{3}{12\pi^2}\int_M ( |W^+|^2 - |W^-|^2)dV_g + \frac{2}{8 \pi^2} 
\int_M \left( |W|^2 - \frac{1}{2} |E|^2 + \frac{1}{24}R^2 \right) dV_g\\
& = \frac{r}{pq}+\frac{q}{pr}+\frac{p}{qr} + \frac{2}{r} + \frac{2}{q} + \frac{2}{p} 
= \frac{(r + p + q)^2}{rpq}, 
\end{split}
\end{align*}
and \eqref{c1} follows. 

 If equality held in \eqref{c1}, then the function 
$u$ in the above argument must be a minimizer of the Yamabe energy, 
so it satisfies the elliptic PDE $-6 \Delta u+ R u=cu^3$ where $c>0$
is a positive constant. By elliptic regularity and the Harnack inequality,
$u$ is a smooth positive function. The metric $g'=u^2g$ has constant
scalar curvature and $\psi'=u^{-3/2}\psi=\psi/|\psi|$ is a $g'$-harmonic spinor
\cite[Theorem 5.24]{LM}. Replacing $g$ and $\psi$ by 
$g'$ and $\psi'$  in the above proof, one may then assume
that $\psi$ is a unit spinor and $g$ has constant scalar curvature.
In the above argument,  all the inequlities used must be equalities. 
In particular $|F^+|= (\sqrt{2}/4)R$ 
and $\langle F^+\cdot \psi,\psi\rangle=-\sqrt{2}|F^+|$. 
Therefore $F^+\cdot\psi=-\sqrt{2}|F^+|\psi$.  
The equation (\ref{boc}) then implies that $\psi$ is parallel,
which implies that $g$ is K\"ahler \cite[Theorem 1.1]{Moroianu}.

Addding $2$ times \eqref{eulerwps2} with $3$ times \eqref{tauwps3} yields
\begin{align}
\frac{1}{4\pi^2}  \int_M  \Big(  2 |W^+|^2 - \frac{1}{2} |E|^2 + \frac{1}{24} R^2 \Big)
dV_g = \frac{ (r + q + p)^2}{r q p}. 
\end{align}
Since $g$ is K\"ahler, using \eqref{wplus}, it follows that 
\begin{align}
\frac{1}{32 \pi^2}  \int_M R^2 dV_g = \frac{ (r + q + p)^2}{r q p} + \frac{1}{ 8 \pi^2} 
\int_M |E|^2 dV_g.
\end{align}
This implies that $g$ is also Einstein, since $g$ attains the maximal value
of the Yamabe energy in \eqref{c1}. 
Thus $g$ is K\"ahler-Einstein, and this contradicts Theorem \ref{noKE},
unless $p =1$. 
\end{proof}

The next lemma will be used in both the proofs of Theorems \ref{t4}
and \ref{crithm}. 
\begin{lemma} 
\label{noEinstein}
Let $g$ be the canonical Bochner-K\"ahler 
metric on $\mathbb{CP}^2_{(r,q,p)}$. If $p \geq r + q$ then 
there is no Einstein metric in the conformal class of $g$. 
\end{lemma}
\begin{proof}
To begin, it is shown in \cite[(2.32)]{DavidGauduchon} that
with the scaling so that
\begin{align}
\mathrm{Vol}(g) = \frac{\pi^2}{2} \frac{1}{pqr},
\end{align}
the scalar curvature of $g$ is given by 
\begin{align}
\label{rgf}
R_g = 24 \big( r ( -r + q + p) |u_1|^2 + q ( r - q + p) |u_2|^2 + p ( r + q - p) |u_3|^2 \big),
\end{align}
where $(u_1,u_2, u_3)$ are coordinates on the Sasakian sphere $S^5 \subset \CC^3$. 
Consequently, in the case $p = q + r$, 
\begin{align}
\label{scal}
R_g =  48 rq ( |u_1|^2 + |u_2|^2),
\end{align} 
which is positive except at the single point $[0,0,1]$ (the 
orbifold point with group of order $p$). The metric $\tilde{g} = R_g^{-2} g $ is 
Ricci-flat. 

Since there are two Einstein metrics in the conformal class, the complete 
manifold $(M \setminus [0,0,1], \tilde{g})$ admits a nonconstant solution of the equation
\begin{align}
\nabla^2 \phi = \frac{\Delta \phi}{m} \tilde{g},
\end{align} 
which is called a concircular scalar field, 
and complete manifolds which admit a non-zero solution were
classified by Tashiro \cite{Tashiro} (see also \cite{Kuhnel}),
who showed that $(X,g)$ must be conformal to one of the following:
\begin{itemize}
\item (A) A direct product $V \times J$, where $V$ is an 
$(m-1)$-dimensional complete Riemannian manifold
and $J$ is an interval,
\item (B) Hyperbolic space $\mathcal{H}^m$, 
\item (C) the round sphere $S^m$.
\end{itemize}

 If $M \setminus [0,0,1]$ were diffeomorphic to a product, 
then any element in $H_2(M)$ would have zero self-intersection. 
However, from the determination of the cohomology ring 
of weighted projective spaces in \cite{KawasakiWPS}, this cannot happen, 
so case (A) is ruled out. 
Cases (B) and (C) cannot happen since $g$ is obviously not locally 
conformally flat. This is a contradiction, and the nonexistence is proved.

In the case $p > q + r$, from $\eqref{rgf}$, the scalar curvature vanishes 
along a hypersurface which divides $M$ into two components $U_+$ and 
$U_-$, with $U_-$ containing the orbifold point $[0,0,1]$ and 
$U_+$ containing the other two orbifold points $[1,0,0]$ and $[0,1,0]$. 
On $U_{\pm}$, the metric $\tilde{g} = R_g^{-2} g$ is 
complete Einstein with negative Einstein constant. 
If there were an Einstein 
metric in the conformal class of $g$, then $U_{\pm}$ would admit a 
concircular scalar field, and the same argument above rules out this 
possibility.
\end{proof}

Next, the lower estimate in Theorem \ref{t4} is given by the following.
\begin{proposition}
\label{oyi2}
 Let $g$ be the canonical Bochner-K\"ahler 
metric on $\mathbb{CP}^2_{(r,q,p)}$. If
\begin{align}
\label{in32}
p < ( \sqrt{r} + \sqrt{q})^2, 
\end{align}
then 
\begin{align}
\label{coyf2}
Y_{\mathrm{orb}}( \mathbb{CP}^2_{(r,q,p)}, [g]) \geq 4 \pi \sqrt{6} 
\sqrt{ \frac{2}{r} + \frac{2}{q} + \frac{2}{p}
- \frac{r}{pq} -\frac{q}{pr}- \frac{p}{qr}} \ .
\end{align}
Furthermore, if $r + q \leq p < ( \sqrt{r} + \sqrt{q})^2$
then strict inequality holds in \eqref{coyf2}. 

If $p < r +q$ then
\begin{align}
\label{coyf2e}
Y_{\mathrm{orb}}( \mathbb{CP}^2_{(r,q,p)}, [g]) = 4 \pi \sqrt{6} 
\sqrt{ \frac{2}{r} + \frac{2}{q} + \frac{2}{p}
- \frac{r}{pq} -\frac{q}{pr}- \frac{p}{qr}} \ .
\end{align}
\end{proposition}

\begin{proof}
Since $W^-(g) = 0$, any metric $\hat{g}$ conformal to 
$g$ also satisfies $W^-(\hat{g}) = 0$. 
Formula \eqref{htwps} above becomes 
\begin{align}
\label{htwps2}
\frac{1}{4\pi^2} \int_M \Big(  - \frac{1}{2}|E|^2 + \frac{1}{24}R^2 \Big)
dV_{\hat{g}} 
= \frac{2}{r} + \frac{2}{q} + \frac{2}{p} - \frac{r}{pq} -\frac{q}{pr} - \frac{p}{qr},
\end{align}
for any metric $\hat{g}$ in the conformal class of $g$. 
If $p < r + q$, the Bochner-K\"ahler 
metric $g$ on $\mathbb{CP}^2_{(r,q,p)}$ is conformal to a positive self-dual 
Einstein metric. Using the fact that an Einstein metric achieves
the Yamabe invariant in its conformal class \cite{Obata}, 
the equality in \eqref{coyf2e} follows.

Next, consider the case 
\begin{align}
\label{case}
r + q \leq p < ( \sqrt{r} + \sqrt{q})^2.
\end{align}
Rewriting \eqref{htwps2}, 
\begin{align}
\begin{split}\label{lyes}
\frac{1}{4\cdot 24 \cdot \pi^2} 
\int_{M} R^2 dV_{\hat{g}}
& =  \frac{2}{r} + \frac{2}{q} + \frac{2}{p}
- \frac{r}{pq} -\frac{q}{pr}- \frac{p}{qr} + \frac{1}{8\pi^2} 
\int_{M}  |E|^2 dV_{\hat{g}}\\
&\geq  \frac{2}{r} + \frac{2}{q} + \frac{2}{p}
- \frac{r}{pq} -\frac{q}{pr}- \frac{p}{qr}.
\end{split}
\end{align}
Note the important fact that 
\begin{align}
 \frac{2}{r} + \frac{2}{q} + \frac{2}{p}
- \frac{r}{pq} -\frac{q}{pr}- \frac{p}{qr} > 0,
\end{align}
precisely when $p < ( \sqrt{r} + \sqrt{q})^2$, this was the inequality 
above in the proof of Theorem~\ref{wpsthm}. 
Furthermore, 
the orbifold conformal Yamabe invariant of $[g]$ is positive;
this follows from \cite[equation (2.37)]{DavidGauduchon} which implies that 
\begin{align}
\mathrm{Vol}(g)^{-1/2}\int_M R_g dV_g= 4 \pi \sqrt{2} \frac{r + q + p}{\sqrt{rqp}} > 0, 
\end{align}
together with \cite[Lemma 3.4]{AB2}.
In contrast to the case of smooth manifolds, one is not assured that 
there is a solution to the orbifold Yamabe problem. 
So to proceed, assume by contradiction that 
\begin{align}
\label{conti}
Y_{\mathrm{orb}}(M,[g]) <  4 \pi \sqrt{6} \sqrt{ \frac{2}{r} + \frac{2}{q} + \frac{2}{p}
- \frac{r}{pq} -\frac{q}{pr}- \frac{p}{qr}}.
\end{align}
If $p < ( \sqrt{r} + \sqrt{q})^2$, 
then the inequality
\begin{align}
\label{nineq}
 4 \pi \sqrt{6} \sqrt{ \frac{2}{r} + \frac{2}{q} + \frac{2}{p}
- \frac{r}{pq} -\frac{q}{pr}- \frac{p}{qr}} < \frac{8 \pi \sqrt{6}}{\sqrt{p}}. 
\end{align}
is satisfied. To see this,  
squaring both sides of \eqref{nineq} results in 
\begin{align}
 \frac{2}{r} + \frac{2}{q} + \frac{2}{p}
- \frac{r}{pq} -\frac{q}{pr}- \frac{p}{qr} < \frac{4}{p}.
\end{align}
Multiplying by $pqr$, and rearranging, this inequality is equivalent to 
\begin{align}
r^2 + q^2 + p^2 - 2 pq - 2 pr + 2qr > 0.
\end{align}
But the left-hand side is a perfect square, 
\begin{align}
r^2 + q^2 + p^2 - 2 pq - 2 pr + 2qr = (p - (r+q))^2
\end{align}
which is strictly positive since $p < r + q$. 

Therefore, by \cite[Theorem 5.2]{AB1} or \cite[Theorem 3.1]{Akutagawacoyi}, 
there exists a solution to the orbifold Yamabe problem which has
constant scalar curvature. Choosing $\hat{g}$ to be this 
Yamabe minimizer, the inequality \eqref{lyes} is then 
\begin{align}
\frac{1}{4\cdot 24 \cdot \pi^2} (Y_{\mathrm{orb}}(M, [g]))^2 \geq  \frac{2}{r} + \frac{2}{q} + \frac{2}{p}
- \frac{r}{pq} -\frac{q}{pr}- \frac{p}{qr},
\end{align}
which contradicts \eqref{conti} and therefore \eqref{coyf2} holds.  

Finally, if equality holds in the inequality \eqref{lyes}, then 
$\hat{g}$ is Einstein. But Lemma~\ref{noEinstein} says 
there is no global Einstein metric in the conformal class of $g$ for $p \geq r + q$,
so strict inequality must hold in \eqref{coyf2} when $p \geq r + q$. 
\end{proof}
\begin{proof}[Proof of Theorem $\ref{t4}$]
This clearly follows from Propositions \ref{gthm} and \ref{oyi2}.
\end{proof}
\begin{proof}[Proof of Theorem $\ref{ulthm}$]
The inequality
\begin{align}
 4 \pi \sqrt{2} \frac{(r + q + p)}{\sqrt{rqp}} < \frac{8 \pi \sqrt{6}}{\sqrt{p}}
\end{align}
is equivalent to 
\begin{align}
(r + q + p)^2 < 12 rq.
\end{align}
Rewrite this 
\begin{align}
\label{xyineq}
(x + y + 1)^2 < 12 x y, 
\end{align}
where $x =r/p$ and $y = q/p$. 
Since $1 \leq r \leq q \leq p$, one must  
determine the region where the inequality \eqref{xyineq} is 
satisfied in the triangle $V = ([0,1] \times [0,1]) \cap \{ y \geq x\}$. 
The level set $(x + y + 1)^2 = 12 xy$ is a convex curve 
in this region, so lies below the line connecting its endpoints 
on the boundary.  It is easy to verify that this line is given by 
\begin{align}
y = \left(1 + \frac{2}{\sqrt{3}}  \right) (1 - x). 
\end{align}
The inequality \eqref{xyineq} is then satisfied for points above 
this line. Converting back to the original variables, this is 
\begin{align}
\frac{q}{p} > \left(1 + \frac{2}{\sqrt{3}}\right) \left( 1 - \frac{r}{p}\right)
\end{align}
which is equivalent to 
\begin{align}
\label{pin2}
p < (2 \sqrt{3} - 3) q + r \sim 0.464 q + r.
\end{align}
Finally, if $p < (2 \sqrt{3} - 3) q + r$, then $p < q + r$,
so \eqref{in3} is satisfied, and the lower estimate  \eqref{oyf2} holds also. 

\end{proof}

\begin{proof}[Proof of Theorem $\ref{crithm}$]
As noted above in the proof of Lemma \ref{noEinstein}, in the 
case $p = q + r$, 
\begin{align}
\label{scal2}
R_g =  48 rq ( |u_1|^2 + |u_2|^2),
\end{align} 
which is positive except at the single point $[0,0,1]$ (the 
orbifold point with group of order $p$).

Assume by contradiction that $\hat{g}$ is a constant scalar curvature 
metric on $M = \mathbb{CP}^2_{(r,q,p)}$ in the conformal class 
of the Bochner-K\"ahler metric $g$.
Letting $E$ denote the traceless Ricci tensor, 
since $\tilde{g} = R_g^{-2} g$ is Ricci-flat, it follows that 
\begin{align}
E_{\hat{g}} = \phi^{-1} \big( - 2\nabla^2 \phi + (\Delta \phi/2) \tilde{g} \big),
\end{align} 
where $\tilde{g} = \phi^{-2} \hat{g}$, and 
the covariant derivatives are taken with respect to $\hat{g}$.
Next, using the argument of Obata \cite{Obata} by integrating on $M$ it follows that
\begin{align}
\begin{split}
\int_{M} \phi | E_{\hat{g}}|^2 d\hat{V} & = \int_{M} \phi E_{\hat{g}}^{ij} 
\left\{  \phi^{-1} \big( -2 \nabla^2 \phi + (\Delta \phi/2) \tilde{g} \big)_{ij} \right\}  d\hat{V}\\
& =   -2\int_{M}  E_{\hat{g}}^{ij} \nabla_i \nabla_j \phi  d\hat{V}
=-2 \lim_{\epsilon \rightarrow 0} \int_{M \setminus B([0,0,1], \epsilon)}   
E_{\hat{g}}^{ij} \nabla_i \nabla_j \phi d\hat{V}.
\end{split}
\end{align}
Since $\tilde{g} = R_g^{-2} g = \phi^{-2} \hat{g}$, and 
$\hat{g}$ and $g$ are related by a strictly positive conformal 
factor, it follows from \eqref{scal2} that 
$\phi \sim R_g \sim \rho^{2}$ as $\rho \rightarrow 0$, 
where $\rho$ is the distance to $[0,0,1]$ with respect to the metric $\hat{g}$. 
Integration by parts yields 
\begin{align}
\label{ibpf}
\int_{M} \phi | E_{\hat{g}}|^2 d\hat{V} 
& = -2 \lim_{\epsilon \rightarrow 0} \left( \int_{ \partial  B([0,0,1], \epsilon)} 
E_{\hat{g}}^{ij} \nabla_i \phi \nu_j d \sigma
- \int_{M \setminus B([0,0,1], \epsilon)}  
(\nabla_j E_{\hat{g}}^{ij} \cdot \nabla_i \phi) d\hat{V} \right).
\end{align}
By the Bianchi identity, 
the second term on the right-hand side is zero since the scalar curvature 
of $\hat{g}$ is constant.  By 
\cite[Theorem 6.4]{TV2}, $\hat{g}$ is a smooth 
Riemannian orbifold, which implies that the curvature is bounded near $[0,0,1]$.
Since $|\nabla \phi| \sim \rho$ near $[0,0,1]$, 
the first term on the right-hand side of \eqref{ibpf} therefore limits to zero
as $\epsilon \rightarrow 0$. 
Consequently, $E_{\hat{g}} \equiv 0$, and $\hat{g}$ is Einstein. 
This is ruled out by Lemma \ref{noEinstein}. 
\end{proof}
\begin{remark} {\em
In the case $p > r + q$, 
there is a complete conformal Einstein metric away from the zero set of the 
scalar curvature, which is a hypersurface. 
The above Obata argument does not work 
in this case to prove that a possible Yamabe minimizer must be Einstein. 
Indeed, there are many known examples of Bach-flat extremal 
K\"ahler metrics which are conformal to complete Einstein metrics 
away from a hypersurface on smooth manifolds (see for example \cite{Tonn}). 
There is a Yamabe minimizer in any such conformal class 
by the solution of the Yamabe problem on smooth manifolds \cite{Schoen},
which in these examples is easily seen to be a non-Einstein metric. 
}
\end{remark}

\def\cprime{$'$}
\providecommand{\bysame}{\leavevmode\hbox to3em{\hrulefill}\thinspace}
\providecommand{\MR}{\relax\ifhmode\unskip\space\fi MR }
\providecommand{\MRhref}[2]{%
  \href{http://www.ams.org/mathscinet-getitem?mr=#1}{#2}
}
\providecommand{\href}[2]{#2}

\end{document}